\documentclass[12pt]{article}
\usepackage{amsthm, amsmath, amssymb, amsfonts, latexsym}
\date{}

\title{On symmetric systems of transport equations}
\author{Evgeny Yu. Panov \\
St. Petersburg Department of V.\,A.~Steklov Institute \\ of Mathematics,
St. Petersburg, Russia; \\
Yaroslav-the-Wise Novgorod State University, \\ Veliky Novgorod, Russia}

\theoremstyle{plain}
\newtheorem{theorem}{Theorem}[section]
\newtheorem{lemma}{Lemma}[section]

\theoremstyle{definition}
\newtheorem{definition}{Definition}[section]
\newtheorem{remark}{Remark}[section]
\newtheorem{example}{Example}[section]
\numberwithin{equation}{section}

\newcommand{\R}{{\mathbb R}}

\newcommand{\D}{{\mathcal D}}
\newcommand{\const}{\mathrm{const}}
\newcommand{\supp}{\mathop{\rm supp}}
\newcommand{\codim}{\mathop{\rm codim}}
\renewcommand{\div}{\mathop{\rm div}\nolimits}
\renewcommand{\Im}{\mathop{\rm Im}}
\begin{document}
\maketitle

\begin{abstract}
We study a symmetric system of transport equations with solenoidal coefficients. This system reduces to an evolutionary equation with a skew-symmetric spatial operator in the real Hilbert space of square-integrable vector-functions, and by general results we claim that there always exists a generalized solution of the Cauchy problem. Uniqueness of this solution is equivalent to skew-adjointness of the spatial operator. We demonstrate that in the case of locally Lipschitz coefficients satisfying a linear growth condition the spatial transport operator is indeed skew-adjoint. For scalar transport equation this result remains true under the weaker DiPerna-Lions conditions.
\end{abstract}

\section{Introduction}
\small
In the half-space $\Pi=\R_+\times\R^n$, where $\R_+=(0,+\infty)$, we consider Cauchy problem for a symmetric system of transport equations
\begin{equation}\label{1}
u_{it}+\sum_{k=1}^n\sum_{j=1}^m a^k_{ij}(x)u_{jx_k}=0, \quad i=1,\ldots,m,
\end{equation}
with the initial conditions
\begin{equation}\label{2}
u_i(0,x)=u_{0i}(x), \quad i=1,\ldots,m.
\end{equation}
The coefficients $a^k_{ij}(x)\in L^2_{loc}(\R^n)$ are supposed to satisfy the symmetricity condition  $a^k_{ij}(x)=a^k_{ji}(x)$, $i,j=1,\ldots,m$, $k=1,\ldots,n$, and the solenoidality condition
\begin{equation}\label{sol}
\sum_{k=1}^n (a^k_{ij})_{x_k}(x)=0
\end{equation}
in the sense of distributions on $\R^n$ (in $\D'(\R^n)$).

In the case of regular coefficients having bounded derivatives up to the second order, the notion of weak solution (belonging to a certain Sobolev class) to problem \eqref{1}, \eqref{2} is well-known, see \S~7.3.2 in Evans' book \cite{Evans}, where existence and uniqueness of weak solutions are proved. In this regular case the solenoidality
condition is not required. In the present paper we discuss the case of rough coefficients and introduce the weaker notion of generalized solution (understood in the distributional sense), so the solenoidality condition plays an important role in our analysis.

We introduce the symmetric matrices $A^k$, $k=1,\ldots,n$, with entries $a^k_{ij}(x)$, $i,j=1,\ldots,m$. In view of solenoidality condition (\ref{sol}) we can write system \eqref{1} in the conservative form
$$
u_t+\sum_{k=1}^n (A^k(x)u)_{x_k}=0, \quad u=(u_1,\ldots,u_m),
$$
and define the notion of a generalized solution (g.s.) of problem \eqref{1}, \eqref{2}.

\begin{definition}\label{def1}
A vector-function  $u=(u_1(t,x),\ldots,u_m(t,x))\in L^\infty(\R_+,L^2(\R^n,\R^m))$ ia called a g.s. to problem \eqref{1}, \eqref{2} if for each test vector $f=f(t,x)\in C_0^1([0,+\infty)\times\R^n,\R^m)$
\[
\int_{\Pi} \bigl[u\cdot f_t+\sum_{k=1}^n A^k(x)u\cdot f_{x_k}\bigr]dtdx+ \int_{\R^n}u_0(x)\cdot f(0,x)dx=0.
\]
Here $u_0(x)=(u_{01}(x),\ldots,u_{0m}(x))\in L^2(\R^n,\R^m)$ is the initial vector-function, and $p\cdot q$ denotes the scalar multiplication of vectors $p,q\in\R^m$.
\end{definition}

We define an unbounded linear operator $S_0$ in the real Hilbert space $H=L^2(\R^n,\R^m)$ by the equality
\[S_0u=\sum_{k=1}^n A^k(x)u_{x_k}(x), \quad u\in D(S_0)=C_0^1(\R^n,\R^m).\]
By simmetricity and solenoidality of the coefficients, for all vectors $u,v\in D(S_0)$
\begin{align*}
(S_0u,v)_H+(u,S_0v)_H=\int_{\mathbb{R}^n}\sum_{k=1}^n\sum_{i,j=1}^m a^k_{ij}(x)(u_{jx_k}v_i+u_j v_{ix_k})dx= \\
\sum_{i,j=1}^m\int_{\R^n}\sum_{k=1}^n a^k_{ij}(x)(u_jv_i)_{x_k}dx=0.
\end{align*}
Hence, the operator $S_0$ is skew-symmetric. In particular, this operator admits a closure $S$. Let $S^*=S_0^*$ be the adjoint operator of $S$. Then, by the skew-symmetricity, $-S\subset S^*$. It is rather well-known (see for instance \cite{Phil,Kr}) that a skew-symmetric operator $-S$ admits a maximal dissipative extension $G$. Moreover, $-S\subset G\subset S^*$ (cf. \cite[Lemma 1.1.5]{Phil}) and the operator $G$ generates the contractive $C_0$-semigroup $T_t=e^{tA}$, $t\ge 0$, on $H$. By \cite[Theorem~2.1]{Zap} the trajectory $u(t,x)=(T_tu_0)(x)$, $u_0\in H$, is a g.s. to problem \eqref{1}, \eqref{2}. In particular, a g.s. of this problem always exists. Moreover, for the constructed g.s. the energy $\|u(t,\cdot)\|_H$ decreases in time. As was also demonstrated in \cite[Corollary~3.1]{Zap}, uniqueness of g.s. (for both forward and backward Cauchy problem)
holds if and only if the operator $S$ is skew-adjoint, i.e., $-S=S^*$. Besides, in this case, $G=-S$, and the operators $e^{-tS}$, $t\in\R$, form a $C_0$-group of orthogonal operators. In particular, the unique g.s. $u(t,x)=(e^{-tS}u_0)(x)$ satisfies the conservation of energy property $\|u(t,\cdot)\|_H=\|u_0\|_H$.

In this paper we are going to prove that under some additional assumptions on the coefficients, the operator $S$ is indeed skew-adjoint, and this implies well-posedness of our problem.

Notice that in the vectorial case $m>1$ neither the method of characteristics nor the renormalization technique can be applied, and we are not able to utilize approaches of DiPerna and Lions \cite{DiL} known for the scalar case $m=1$. Instead, we use functional-analytic methods described above.

\section{Assumptions and preliminaries}\label{sec1}

We suppose that the coefficients $a^k_{ij}(x)\in W_{\infty,loc}^1(\R^n)$ (that is, they are locally Lipschitz functions) and have
at most linear growth at infinity. The latter condition can be written in the form
\begin{equation}\label{gr}
\|A^k(x)\|\le c(1+|x|), \quad k=1,\ldots,n, \ c=\const,
\end{equation}
where $\|A\|$ is the operator norm of a matrix $A$, and $|x|$ denotes the Euclidean norm of a finite-dimensional vector $x$.

By the definition of adjoint operator, $D(S^*)$ consists of such $u\in H$ that for all $f\in D(S_0)=C_0^1(\R^n,\R^m)$ the relation $(u,Sf)_H=(v,f)_H$ holds for some $v\doteq S^* u\in H$. Revealing the relation $(u,Sf)_H=(v,f)_H$, we obtain that
for all $f\in C_0^1(\R^n,\R^m)$
\begin{equation}\label{3}
\int_{\R^n} \sum_{k=1}^n u(x)\cdot A^k(x)f_{x_k}(x)dx=\int_{\R^n} v(x)\cdot f(x)dx.
\end{equation}
Since $u(x)\cdot A^k(x)f_{x_k}(x)=A^k(x)u(x)\cdot f_{x_k}(x)$ by symmetricity of the matrices $A^k$, relation (\ref{3}) means that
\begin{equation}\label{S*}
S^*u=v=-\sum_{k=1}^n (A^k(x)u(x))_{x_k} \ \mbox{ in } \D'(\R^n).
\end{equation}
Hence, the domain $D(S^*)$ consists of such $u(x)\in H$ that the distribution $\displaystyle v=-\sum_{k=1}^n (A^k(x)u(x))_{x_k}$ lies in $H$, and we set the value $S^*u=v$ as indicated in (\ref{S*}).

We denote by $H_c$ the space of vector-functions in $H$ with compact support.

\begin{lemma}\label{lem1}
The space $H_c\cap D(S^*)$ is a core of $S^*$, that is, it is dense in the space $D(S^*)$ equipped with the graph norm of $S^*$.
\end{lemma}

\begin{proof}
Assume that $u(x)\in D(S^*)$. We need to approximate $u(x)$ in the graph norm of $D(S^*)$ by compactly supported vector-functions in $D(S^*)$. Let $p(y)\in C_0^1(\R^n)$ be such a function that $p(y)=1$ for $|y|\le 1$, $p(y)=0$ for $|y|\ge 2$. We set $u_R(x)=p(x/R)u(x)$, where $R>1$. It is clear that $u_R\in H_c$.
Since $p(x/R)\in C_0^1(\R^n)$, the distribution $\displaystyle v_R\doteq -\sum_{k=1}^n (A^k(x)p(x/R)u(x))_{x_k}$ can be represented as
\begin{equation}\label{4}
v_R=-p(x/R)\sum_{k=1}^n (A^k(x)u(x))_{x_k}-\frac{1}{R}\sum_{k=1}^n p_{y_k}(x/R)A^k(x)u(x)=p(x/R)v(x)-w_R(x),
\end{equation}
where $v(x)=S^*u(x)\in H$, $\displaystyle w_R(x)=\frac{1}{R}\sum_{k=1}^n p_{y_k}(x/R)A^k(x)u(x)\in H$. We see that $v_R\in H$ and therefore $u_R\in D(S^*)$, $S^*u_R=v_R$ (see the description of $D(S^*)$ above). Thus, $u_R\in H_c\cap D(S^*)$. It is clear that $u_R=p(x/R)u(x)\to u(x)$, $p(x/R)v(x)\to v(x)$ in $H$ as $R\to\infty$. Let us show that $w_R(x)\to 0$ in $H$
as $R\to\infty$. Notice that $\supp \nabla_y p(x/R)$ lie in the annulus $R\le |x|\le 2R$ while $\|A^k(x)\|\le c(1+2R)\le 3cR$ in this annulus, by condition (\ref{gr}). Therefore,
\[
|w_R(x)|\le 3cn\|\nabla p\|_\infty |u(x)|\theta(|x|-R),
\]
where $\theta(s)$ is the Heaviside function. This estimate implies the desired relation $w_R\to 0$ in $H$ as $R\to\infty$.
It follows that $v_R\mathop{\to}\limits_{R\to\infty} v$ in $H$. Hence, $u_R\to u$, $S^*u_R\to S^*u$ in $H$, which means
convergence $u_R\to u$ in the graph norm. The proof is complete.
\end{proof}

We need a variant of the famous DiPerna-Lions commutation lemma
\cite[Lemma~II.1]{DiL}. Suppose $\rho(z)\in C_0^1(\R^n)$, $\supp\rho\subset B_1(0)=\{z\in\R^n | |z|\le 1\}$, $\rho(z)\ge 0$, $\int_{\R^n}\rho(z)dz=1$. For $u(x)\in L^1_{loc}(\R^n)$ we introduce the corresponding averaged functions
\[u_h(x)=u*\rho_h(x)=\int_{\R^n}\rho_h(x-y)u(y)dy,\]
where $h\in (0,1)$, $\rho_h(z)=h^{-n}\rho(z/h)$.

\begin{lemma}\label{lmDL}
Let $a(x)\in C(\R^n)$ satisfy the Lipschitz condition
${|a(x)-a(y)|\le L|x-y|}$ for all $x,y\in\R^n$; $u(x)\in L^p(\R^n)$, $1\le p<\infty$. Then
\begin{equation}\label{DL}
\frac{\partial}{\partial x_k}(au_h-(au)_h)(x)\mathop{\to}_{h\to 0} 0 \ \mbox{ in } L^p(\R^n)
\end{equation}
for all $k=1,\ldots,n$.
\end{lemma}

\begin{proof}
First of all we notice that by the Lipschitz condition the generalized derivatives $a_{x_k}(x)\in L^\infty(\R^n)$, $k=1,\dots,n$, and $\|\nabla a(x)\|_\infty\le L$. Since \[(au_h-(au)_h)(x)=h^{-n}\int_{\R^n}(a(x)-a(y))\rho((x-y)/h) u(y)dy,\]
there exist the generalized derivatives
\begin{align}\label{rel}
\frac{\partial}{\partial x_k}(au_h-(au)_h)(x)=h^{-n}\int_{\R^n}a_{x_k}(x)\rho((x-y)/h) u(y)dy+\nonumber\\ h^{-n-1}\int_{\R^n}(a(x)-a(y))\rho_{z_k}((x-y)/h) u(y)dy=I_1^h(x)+I_2^h(x).
\end{align}
The first term in this sum
\begin{equation}\label{rel1}
I_1^h(x)=a_{x_k}(x)\int_{\R^n}\rho_h(x-y) u(y)dy=a_{x_k}(x)u_h(x)\to a_{x_k}(x)u(x)
\end{equation}
as $h\to 0$ in $L^p(\R^n)$ because $u_h\mathop{\to}\limits_{h\to 0} u$ in $L^p(\R^n)$ by the known property of averaged functions while the derivative $a_{x_k}(x)\in L^\infty(\R^n)$.
Next, we estimate the term $I_2^h(x)$. Obviously,
\begin{align}\label{es1}
|I_2^h(x)|\le h^{-n-1}\int_{\R^n}|a(x)-a(y)||\rho_{z_k}((x-y)/h)| |u(y)|dy\le \nonumber\\ \omega^h(x)\doteq Lh^{-n}\int_{\R^n} h^{-1}|x-y||\rho_{z_k}(h^{-1}(x-y))| |u(y)|dy.
\end{align}
By the properties of averaged functions (with the kernel $|z||\rho_{z_k}(z)|$) $\omega^h(x)\in L^p(\R^n)$ and
\begin{equation}\label{rel2}
\omega^h(x)\to C|u(x)|
\end{equation}
as $h\to 0$ both in $L^p(\R^n)$ and a.e. in $\R^n$. Here $\displaystyle C=L\int_{\R^n} |z||\rho_{z_k}(z)|dz=\const$.
Now, let $x$ be a common Lebesgue point of the vector $\nabla a(y)$ and the function $u(y)$. Then
\begin{align}\label{rel2-1}
I_2^h(x)=h^{-n-1}\int_{\R^n}(a(x)-a(y))\rho_{z_k}(h^{-1}(x-y))(u(y)-u(x))dy+\nonumber\\ u(x)h^{-n-1}\int_{\R^n}(a(x)-a(y))\rho_{z_k}(h^{-1}(x-y))dy.
\end{align}
The first term in the right-hand part of (\ref{rel2-1}) is estimated as
\begin{align}\label{est2-1}
h^{-n-1}\left|\int_{\R^n}(a(x)-a(y))\rho_{z_k}(h^{-1}(x-y))(u(y)-u(x))dy\right|\le \nonumber\\
Lh^{-n}\int_{\R^n}h^{-1}|x-y||\rho_{z_k}(h^{-1}(x-y))||u(y)-u(x)|dy\mathop{\to}_{h\to 0} 0
\end{align}
since $x$ is a Lebesgue point of $u(y)$. We introduce the function ${\displaystyle J^h(x)=h^{-n-1}\int_{\R^n}(a(x)-a(y))\rho_{z_k}(h^{-1}(x-y))dy}$ and represent it in the form
\begin{align}\label{rel2-2}
J^h(x)=h^{-n-1}\int_{\R^n}(a(x)-a(y)-\nabla a(x)\cdot(x-y))\rho_{z_k}(h^{-1}(x-y))dy+\nonumber\\ h^{-n-1}\int_{\R^n}\nabla a(x)\cdot (x-y)\rho_{z_k}(h^{-1}(x-y))dy=\nonumber\\
h^{-n-1}\int_{\R^n}\int_0^1 (\nabla a(x+s(y-x))-\nabla a(x))\cdot(x-y)\rho_{z_k}(h^{-1}(x-y))dsdy+\nonumber\\ \nabla a(x)\cdot\int_{\R^n}z\rho_{z_k}(z)dz.
\end{align}
We utilize here the identity
\[
a(x)-a(y)=\int_0^1\nabla a(x+s(y-x))\cdot(x-y)ds,
\]
which holds for a.e. $y\in\R^n$.
To estimate the first term in the right-hand part of (\ref{rel2-2}), we make the change of variables $y\to z=s(x-y)$, resulting in
\begin{align}\label{est2-2}
h^{-n-1}\left|\int_{\R^n}\int_0^1 (\nabla a(x+s(y-x))-\nabla a(x))\cdot(x-y)\rho_{z_k}(h^{-1}(x-y))dsdy\right|\le\nonumber\\
h^{-n}\int_{\R^n} |\nabla a(x-z)-\nabla a(x)|\rho_1(z/h)dz,
\end{align}
where we denote
\[
\rho_1(y)=|y|\int_0^1s^{-n-1}|\rho_{z_k}(y/s)|ds.
\]
Remark that $\rho_1\ge 0$, $\supp\rho_1\subset B_1(0)$, 
\begin{align*}
\int_{\R^n}\rho_1(y)dy=\int_0^1\int_{\R^n} |s^{-1}y||\rho_{z_k}(s^{-1}y)|s^{-n}dyds= \\ \int_0^1\int_{\R^n}|z||\rho_{z_k}(z)|dz ds=\int_{\R^n}|z||\rho_{z_k}(z)|dz
\end{align*}
(we made the change $z=s^{-1}y$), and since $x$ is a Lebesgue point of the vector $\nabla a(y)$, we find
\[h^{-n}\int_{\R^n} |\nabla a(x-z)-\nabla a(x)|\rho_1(z/h)dz\mathop{\to}_{h\to 0} 0.\]
It follows from (\ref{rel2-2}) and (\ref{est2-2}) that
\begin{equation} \label{rel3}
J^h(x)\mathop{\to}_{h\to 0} \nabla a(x)\cdot\int_{\R^n}z\rho_{z_k}(z)dz=-a_{x_k}(x).
\end{equation}
Here, we apply the integration by parts formula
\[\int_{\R^n}z\rho_{z_k}(z)dz=-\int_{\R^n}\frac{\partial z}{\partial z_k}\rho(z)dz=-e_k, \]
where $e_k$ is the $k$th basis vector in $\R^n$.
In view of (\ref{rel2-1}), (\ref{est2-1}) it follows from (\ref{rel3}) that $I_2^h(x)\mathop{\to}\limits_{h\to 0} -a_{x_k}(x)u(x)$. By our choice, $x\in\R^n$ is an arbitrary point of a set of full Lebesgue measure. Therefore, $I_2^h(x)\to -a_{x_k}(x)u(x)$ as $h\to 0$ a.e. in $\R^n$. Further, we notice that
\begin{align*}
|I_2^h(x)+a_{x_k}(x)u(x)|^p\le 2^{p-1}(|I_2^h(x)|^p+|a_{x_k}(x)u(x)|^p)\le \\ 2^{p-1}((\omega^h(x))^p+|a_{x_k}(x)u(x)|^p).
\end{align*}
The left-hand side of this inequality converges as $h\to 0$ to zero a.e. in $\R^n$, while its right-hand side converges both in  $L^1(\R^n)$ and a.e. in $\R^n$, by relation (\ref{rel2}). Applying Fatou's lemma to the sequence
\[
2^{p-1}((\omega^h(x))^p+|a_{x_k}(x)u(x)|^p)-|I_2^h(x)+a_{x_k}(x)u(x)|^p,
\]
we derive that
\[
\lim_{h\to 0} \int_{\R^n}|I_2^h(x)+a_{x_k}(x)u(x)|^pdx=0,
\]
that is, $\displaystyle I_2^h(x)\mathop{\to}_{h\to 0} -a_{x_k}(x)u(x)$ in $L^p(\R^n)$. This, together with  (\ref{rel}), (\ref{rel1}), gives the required relation (\ref{DL}). The proof is complete.
\end{proof}

\section{Main results}\label{secM}

\begin{theorem}\label{thM}
Assume that the coefficients $a^k_{ij}(x)\in W^1_{\infty,loc}(\R^n)$ and satisfy the growth condition $(1+|x|)^{-1}a^k_{ij}(x)\in L^\infty(\R^n)$. Then the operator $S$ is skew-adjoint.
\end{theorem}

\begin{proof}
We are going to demonstrate that the space $H_c\cap D(S^*)\subset D(S)$. Choosing $u=u(x)\in H_c\cap D(S^*)$, we consider the sequence of averaged vector-functions $u_h(x)=u*\rho_h(x)\in C_0^1(\R^n,\R^m)=D(S_0)$. If $\supp u(x)$ is contained in a ball $B_R(0)$ then $\supp u_h\subset B_{R+1}(0)$ for all $h\in (0,1)$ and the vectors $v=S^*u$, $v^h\doteq S^* u_h=-S_0u_h$ do not depend on matrices $A^k(x)$ for $|x|>R+1$. By our assumptions these matrices are Lipschitz continuous in $B_{R+1}(0)$ and can be extended to Lipschitz continuous symmetric matrices on the whole space $\R^n$. Hence, without loss of generality we may initially suppose that the matrices $A^k(x)$ satisfy the global Lipschitz condition in $\R^n$.
It then follows from Lemma~\ref{lmDL} that for all $i=1,\ldots,m$
\[(v^h-v_h)_i(x)=\sum_{k=1}^n \sum_{j=1}^m \frac{\partial}{\partial x_k}[(a^k_{ij}u_j)_h(x)-a^k_{ij}(x)(u_j)_h(x)]
\mathop{\to}_{h\to 0} 0 \ \mbox{ in } L^2(\R^n). \]
This implies that $v^h-v_h\to 0$ as $h\to 0$ in the space $H$. On the other hand, $u_h\mathop{\to}\limits_{h\to 0} u$,  $v_h\mathop{\to}\limits_{h\to 0} v$ in $H$ by the property of averaged functions. Hence, $u_h\to u$, $v^h=-Su_h\to v$ as $h\to 0$ in $H$. Since the operator $S$ is closed, we conclude that $u\in D(S)$ and $-Su=v$. We have proved the required inclusion $H_c\cap D(S^*)\subset D(S)$. Since the operator $S$ is closed, the closure of $H_c\cap D(S^*)$ in $D(S^*)$ with respect to the graph norm also lies in $D(S)$. But by Lemma~\ref{lem1} this closure coincides with $D(S^*)$, and we obtain that $D(S^*)\subset D(S)\subset D(S^*)$. Thus, $D(S)=D(S^*)$, therefore $-S=S^*$ and $S$ is a skew-adjoint operator. The proof is complete.
\end{proof}

It follows from Theorem~\ref{thM} and results of \cite{Zap} that the problem \eqref{1}, \eqref{2} is well-posed.

\begin{theorem}\label{wp} Under the assumptions of Theorem~\ref{thM} there exists a unique g.s. $u(t,x)$ of problem \eqref{1}, \eqref{2}. This solution admits the representation $u(t,x)=(e^{-tS} u_0)(x)$ and satisfies the conservation of energy property $\|u(t,\cdot)\|_H=\|u_0\|_H$.
\end{theorem}

We underline that existence of g.s. remains valid in the case $a^k_{ij}(x)\in L^2_{loc}(\R^n)$ (see Introduction).
But, in this  general case there are numerous examples of non-uniqueness of g.s. even for scalar transport equations, see for instance \cite{Aiz,Br,CLR,Dep,PaTr}, so that the operator $S$ may fail to be skew-adjoint, see also Example~\ref{Ex} below.

\begin{remark}\label{rem1}
In the case $m=1$ the problem \eqref{1}, \eqref{2} has the form
\begin{equation}\label{1-1}
u_t+a(x)\cdot\nabla u(x)=0, \quad u(0,x)=u_0(x)\in L^2(\R^n),
\end{equation}
where $a(x)=(a_1(x),\ldots,a_n(x))$ is a solenoidal vector of coefficients. This problem was studied by DiPerna and Lions in their seminal paper \cite{DiL}. It follows from \cite[Corollary~II.1]{DiL} that under the assumptions
\begin{equation}\label{reqDL}
a(x)\in W^1_{2,loc}(\R^n,\R^n), \quad (1+|x|)^{-1}|a(x)|\in L^1(\R^n)+L^\infty(\R^n)
\end{equation}
there exists a unique g.s. of problem \eqref{1}, \eqref{2}. Since the requirements (\ref{reqDL}) remains valid after the change $a\to -a$, the same result is true for the backward Cauchy problem (in the domain $t<0$). According to \cite[Corollary~3.1]{Zap} the operator $S$ is skew-adjoint. We see that in the scalar case $m=1$ the result of Theorem~\ref{thM} holds under much weaker requirements (\ref{reqDL}) than in the vectorial case. This is connected with
possibility in the scalar case to apply the renormalization method.
\end{remark}

Remark also that the growth condition is essential for the statement of Theorem~\ref{thM} even in the scalar case $m=1$.
Let us confirm it by the following simple example.

\begin{example}\label{Ex}
We consider equation \eqref{1-1} with $n=2$, $a=a(x,y)=(x^2,-2xy)$ (notice that $\div a=0$, as required). Let $S$ be a closure of the operator $S_0u=x^2u_x-2xyu_y$. As is easy to verify, for each $v(z)\in C_0^\infty(\R)$ the functions
\begin{equation}\label{ex1}
u(x,y)=\left\{\begin{array}{lcr} v(yx^2)e^{\frac{1}{hx}} & , & hx<0, \\ 0 & , & hx\ge 0 \end{array}\right.
\end{equation}
are solutions of the resolvent equation $u-hS^*u=0$, $h\in\R\setminus\{0\}$. Moreover,
$u(x,y)\in C^\infty(\R^2)\cap L^\infty(\R^2)\cap L^2(\R^2)$, $\|u\|_\infty=\|v\|_\infty$, $\|u\|_2=\|v\|_2\sqrt{|h|/2}$.
By the arbitrariness of $v(z)$, the operators $E-hS^*$ have infinite-dimensional kernels. Therefore, $\codim\Im(E-hS)=\dim\ker(E-hS^*)=\infty$ and the operator $S$ has the same infinite deficiency indexes. Hence, this operator is not skew-adjoint (but it admits infinitely many skew-adjoint extensions).

For the modified coefficients $a=a(x,y)=((x_+)^2,-2x_+y)$, with $x_+=\max(x,0)$, functions
(\ref{ex1}) lie in $\ker(E-hS^*)$ only for $h<0$. If $h>0$, this kernel is trivial. This means that the deficiency indexes $d_\pm=\codim\Im(E\pm S)$ of the operator $S$ are as follows: $d_-=0$, $d_+=\infty$. In particular, $S$ is a maximal skew-symmetric operator, which is not skew-adjoint.
\end{example}

\end{document}